\documentclass{article}
\usepackage{amsfonts}
\usepackage{amsmath}
\usepackage{amsthm}
\usepackage{amssymb}
\usepackage{geometry}
\usepackage{bbm}
\usepackage{mathtools}
\usepackage{babel}
\usepackage{caption}
\title{Eigenfunction Expansion and the Decomposition of Jacobi Operators on $\mathbb{Z}$}
\author{Netanel Levi\footnote{Institute of Mathematics, The Hebrew University of Jerusalem, Jerusalem, 91904, Israel.
		Supported in part by the Israel Science Foundation (Grant No.\ 1378/20) and in part by the United States-Israel Binational Science Foundation
		(Grant No.\ 2020027), Email: netanel.levi@mail.huji.ac.il}}
\numberwithin{equation}{section}
\numberwithin{figure}{section}
\theoremstyle{plain}
\newtheorem{theorem}{\protect\theoremname}[section]
\theoremstyle{definition}
\newtheorem{definition}[theorem]{\protect\definitionname}
\theoremstyle{definition}
\newtheorem*{definition*}{\protect\definitionname}
\theoremstyle{remark}
\newtheorem{remark}[theorem]{\protect\remarkname}
\theoremstyle{plain}
\newtheorem{lemma}[theorem]{\protect\lemmaname}
\theoremstyle{plain}

\theoremstyle{plain}
\newtheorem{prop}[theorem]{\protect\propositionname}
\theoremstyle{plain}
\newtheorem{claim}[theorem]{\protect\claimname}

\DeclareMathOperator{\im}{Im}

\DeclareMathOperator{\tr}{tr}
\DeclareMathOperator{\diag}{diag}
\DeclareMathOperator{\vspan}{sp}

\DeclareMathOperator{\slim}{s-\lim}
\DeclareMathOperator{\wlim}{w-\lim}

\providecommand{\claimname}{Claim}
\providecommand{\corollaryname}{Corollary} 
\providecommand{\definitionname}{Definition}
\providecommand{\lemmaname}{Lemma}
\providecommand{\propositionname}{Proposition}
\providecommand{\remarkname}{Remark}
\providecommand{\theoremname}{Theorem}

\begin{document}
	\maketitle
	\sloppy
	\begin{abstract}
		Let $J$ be a Jacobi operator on $\ell^2\left(\mathbb{Z}\right)$. We prove an eigenfunction expansion theorem for the singular part of $J$ using subordinate solutions to the eigenvalue equation. We exploit this theorem in order to show that $J$ can be decomposed as a direct integral of half-line operators.
	\end{abstract}
	\section{Introduction}\label{intro}
	Let $J$ be a Jacobi operator on $\ell^2\left(\mathbb{Z}\right)$, by which we mean an operator of the form
	\begin{equation}\label{jac_op_line_eq}
		\left(J\psi\right)\left(n\right)=a_{n-1}\psi\left(n-1\right)+a_n\psi\left(n+1\right)+b_n\psi\left(n\right),
	\end{equation}
	where $\left(a_n\right)_{n\in\mathbb{Z}}$ and $\left(b_n\right)_{n\in\mathbb{Z}}$ are real-valued and bounded and $\left(a_n\right)_{n\in\mathbb{Z}}$ is positive. In this work, we are inerested in studying the connection between the characterization of the singular part of $J$ via the notion of subordinacy, and the eigenfunction expansion of $J$. Throughout, we will study solutions to the eigenvalue equation, i.e. functions $u:\mathbb{Z}\to\mathbb{C}$ which satisfy
	\begin{equation}\label{ev_eq_intro}
		a_{n-1}u\left(n-1\right)+a_nu\left(n+1\right)+b_nu\left(n\right)=Eu\left(n\right)
	\end{equation}
	for some $E\in\mathbb{R}$ and for every $n\in\mathbb{Z}$. We say that a Borel measure $\mu$ is in the spectral measure class of $J$ if it is equivalent to $P$, the spectral measure of $J$, namely $P\left(A\right)=0\iff\mu\left(A\right)=0$ for any Borel set $A$ (see Definition \ref{measure_class_def}).
	
	It is well-known (see, e.g.\ \cite[Chapter V, Theorem 2.1	]{Ber}) that for any such $\mu$, there exists a function $N:\mathbb{R}\to\mathbb{N}$ which is defined $\mu$-almost everywhere, and for $\mu$-almost every $E\in\mathbb{R}$ there exists a collection of solutions of (\ref{ev_eq_intro}), $\left(u_E^k\right)_{k=1}^{N\left(E\right)}$, such that for any compactly supported $\psi\in\ell^2\left(\mathbb{Z}\right)$ and for any Borel set $A$,
	\begin{equation}\label{ef_exp_intro}
		P\left(A\right)\psi=\int_\mathbb{R}\sum\limits_{k=1}^{N\left(E\right)}\langle u_E^k,\psi\rangle u_E^kd\mu\left(E\right).
	\end{equation}
	A few remarks are in order
	\begin{enumerate}
		\item Both the choice of solutions and the function $N$ are defined $\mu$-almost everywhere.
		\item $N$ is called the multiplicity function of $J$.
		\item This is a generalization of the finite-dimensional version of the spectral theorem.
		\item We are citing here a very partial version of the eigenfunction expansion theorem. For example, (\ref{ef_exp_intro}) actually holds for a larger class of sequences. We refer the reader to \cite[Chapter V]{Ber} for more information on eigenfunction expansions.
	\end{enumerate}
	It is easy to see that for any $E\in\mathbb{R}$, any solution $u$ of (\ref{ev_eq_intro}) is determined by $u\left(0\right)$ and $u\left(1\right)$, and vice-versa: for any $\alpha_0,\alpha_1\in\mathbb{C}$, there exists a solution $u$ such that $u\left(0\right)=\alpha_0,u\left(1\right)=\alpha_1$. This implies that the space of solutions of (\ref{ev_eq_intro}) is two-dimensional, and so $N\left(E\right)\leq 2$ for any $E\in\mathbb{R}$. Furthermore, if we restrict our attention to the singular part of $J$, then $N\left(E\right)=1$. More formally, in \cite{Kac} it was shown that $N\left(E\right)=1$ for $\mu_s$-almost every $E\in\mathbb{R}$. This was also proved later in \cite{Sim1} using the theory of finite rank perturbations, and in \cite{G2,L} using subordinacy theory, which relates singularity of $\mu$ to asymptotic properties of solutions of (\ref{ev_eq_intro}).
	
	Subordinacy theory was first introduced in \cite{GP}, and was later extended and developed in many ways (for example \cite{DGO,G,G2,JL,JL2,KP,L}). Roughly speaking, a solution to (\ref{ev_eq_intro}) will be called subordinate if it has slower growth than any other solution at both $\pm\infty$. A formal definition is given in Subection \ref{sub_theory_borel_transform_section}. Now, a support to the singular part of $\mu$ can be given in terms of subordinacy in the following way:
	\begin{theorem}\emph{(\cite[Theorem 2.7.11]{DF})}
		Let $\mu_s$ be the part of $\mu$ which is singular w.r.t.\ the Lebesgue measure. Then $\mu_s$ is supported on the set
		\begin{center}
			$S=\left\{E\in\mathbb{R}:\text{ there exists a subordinate solution to (\ref{ev_eq_intro})}\right\}$.
		\end{center}
	\end{theorem}
	For the rest of this section, the measure $\mu$ is taken to be $\mu_{\delta_0}+\mu_{\delta_1}$, where $\mu_\psi$ is the spectral measure of $\psi$ w.r.t.\ $J$. It is not hard to see that $\mu$ is in the spectral measure class of $J$. It is also not hard to show that if a subordinate solution exists, then it is unique (up to scalar multiplication). However, as mentioned before, the space of solutions is two-dimensional, and so since $N\left(E\right)=1$ for $\mu_s$-almost $E\in\mathbb{R}$, it is natural to expect that the solution which is chosen in the RHS of (\ref{ef_exp_intro}) is the subordinate one. Given a function $f:\mathbb{Z}\to\mathbb{R}_{>0}$, we denote
	\begin{center}
		$\ell^2\left(\mathbb{Z};f\right)\coloneqq\left\{u:\mathbb{Z}\to\mathbb{C}:\sum\limits_{n\in\mathbb{Z}}\frac{\left|u\left(n\right)\right|^2}{f\left(n\right)}<\infty\right\}$.
	\end{center}
	Let us also denote by $P_s$ the part of $P$ which is singular w.r.t\ the Lebesgue measure (we refer the reader to Subsection \ref{prelim_spectral_thm} for precise definitions). In this work, we show
	\begin{theorem}\label{main_thm_1_intro}
		Let $\mu\coloneqq\mu_{\delta_0}+\mu_{\delta_1}$. Then there exists a choice, for $\mu_s$-almost every $E\in\mathbb{R}$, of a subordinate solution $u_E$ which satisfies $\left|u_E\left(0\right)\right|^2+\left|u_E\left(1\right)\right|^2=1$, such that for any Borel set $B$ and for any $\psi\in\ell^2\left(\mathbb{Z};n^2\right)$,
		\begin{center}
			$P_s\left(B\right)\psi=\int_{B\cap S}\langle u_E,\psi\rangle u_Ed\mu_s\left(E\right)$.
		\end{center}
	\end{theorem}
	\begin{remark}
		A certain version of Theorem \ref{main_thm_1_intro} remains valid for a general Borel measure $\rho$ which is in the spectral measure class of $J$. The only modification is in the normalization of the subordinate solutions.
	\end{remark}
	We will exploit Theorem \ref{main_thm_1_intro} in order to obtain a decomposition of $J$ into half-line operators. In general, the decomposition of Jacobi operators on graphs has drawn a lot of attention and proved useful in many cases. A very non-comprehensive list of works in that area includes \cite{Br,Br1,BrKe,BrLe,KN,NS,Sol}. Generally, such decompositions are obtained when the underlying operator is the adjacency matrix (namely $a_n=1$ and $b_n=0$ for all $n\in\mathbb{Z}$), and the underlying graph possesses some strong symmetry properties. In these cases, the operator is decomposed into a direct sum of Jacobi operators on $\mathbb{N}$. In this work, we describe a decomposition for general bounded self-adjoint Jacobi operators on $\mathbb{Z}$ with no symmetry assumptions on the potential. This is the first result of this kind as far as we know. The price that is paid for this generality is that the decomposition is no longer given by a direct sum, but as a certain integral of operators.
	
	Let us denote by $J_0^\pm$ the restriction of $J$ to $\mathbb{Z}_\pm$, where $\mathbb{Z}_+=\left\{1,2,\ldots\right\}$ and $\mathbb{Z}_-=\mathbb{Z}\setminus\mathbb{Z}_+$. For any $\theta\in\left[0,\pi\right)$, we denote
	\begin{center}
		$J_\theta\coloneqq\left(J_0^--\cot\theta\langle\delta_0,\cdot\rangle\delta_0\right)\oplus\left(J_0^+-\tan\theta\langle\delta_1,\cdot\rangle\delta_1\right)$.
	\end{center}
	Note that for any $\theta\in\left[0,\pi\right)$, $J_\theta$ is the direct sum of two half-line Jacobi operators. For any self-adjoint operator $T$, we denote by $T_s$ the singular part of $T$ (again, we refer the reader to Subsection \ref{prelim_spectral_thm} for precise definitions). Our first result is the following decomposition theorem:
	\begin{theorem}\label{main_thm_3_intro}
		There exists a projection-valued measure $Q$ which depends on $J$ and which is defined on Borel subsets of $\left[0,\pi\right)$, such that
		\begin{equation}\label{decomp_eq_intro}
			\int_0^\pi J_\theta dQ\left(\theta\right)=J_s.
		\end{equation}
	\end{theorem}
	Theorem \ref{main_thm_3_intro} follows from the following
	\begin{theorem}\label{main_thm_2_intro}
		The measure $Q$ commutes with $J$, and for any $a,b\in\left(0,\frac{\pi}{2}\right)$ there exists $\gamma>0$ such that for any $\alpha,\beta\in\left(a,b\right)$ and for any $\psi\in\ell^2\left(\mathbb{Z};n^2\right)$,
		\begin{equation}
			\|J_sQ\left(\left[\alpha,\beta\right)\right)\psi-J_\theta Q\left(\left[\alpha,\beta\right)\right)\psi\|<\gamma\left|\beta-\alpha\right|\|Q\left(\left[\alpha,\beta\right)\right)\psi\|
		\end{equation}
	\end{theorem}
	\begin{remark}
		\begin{enumerate}
			\item The decomposition (\ref{decomp_eq_intro}) is given by the strong limit of the Riemann-Stieltjes sums.
			\item Theorems \ref{main_thm_3_intro} and \ref{main_thm_2_intro} actually hold when replacing $J$ and $J_\theta$ with $F\left(J\right)$ and $F\left(J_\theta\right)$ for certain functions $F$. The stronger versions of these theorems are given in Theorems \ref{splitting_thm} and \ref{integral_thm}.
		\end{enumerate}
	\end{remark}
	The rest of the paper is structured as follows. In Section $2$, we present some preliminary background in the one-dimensional theory of Jacobi operators and some basic measure-theoretic background. In Section $3$, we present the eigenfunction expansion of Jacobi operators in $\ell^2\left(\mathbb{Z}\right)$ and in particular, prove Theorem \ref{main_thm_1_intro}. In Section $4$, we define the measure $Q$ and the collection $\left(J_\theta\right)_{\theta\in\left[0,\pi\right)}$ mentioned in Theorems \ref{main_thm_2_intro} and \ref{main_thm_3_intro}, and prove these theorems. Section $5$ contains some examples and further discussion.
	
	{\bf Acknowledgments} I would like to thank Jonathan Breuer for useful discussions.
	\section{Preliminaries}
	\subsection{The spectral theorem and operator-valued measures}\label{prelim_spectral_thm}
	Let $\mathcal{H}$ be a Hilbert space and let \mbox{$\Theta:\text{Borel}\left(\mathbb{R}\right)\to\mathcal{B}\left(\mathcal{H}\right)$}. $\Theta$ will be called a {\it nonnegative operator-valued measure} (NOVM), if $\Theta\left(\emptyset\right)=0$, $\Theta\left(B\right)> 0$ for any Borel set $B$, and for any collection of disjoint Borel set $\left(B_j\right)_{j\in\mathbb{N}}$, one has
	\begin{center}
		$\Theta\left(\underset{n\in\mathbb{N}}{\bigcup}B_j\right)=\sum\limits_{j=1}^\infty\Theta\left(B_j\right)$
	\end{center}
	where the RHS converges weakly to the LHS. If  for any two Borel sets $B_1,B_2$, $\Theta\left(B_1\right)$ is an orthogonal projection and in addition, $\Theta\left(B_1\cap B_2\right)=\Theta\left(B_1\right)\Theta\left(B_2\right)$, then we will say that $\Theta$ is a {\it projection-valued measure} (PVM). We will also be interested in NOVMs which have {\it bounded trace}.
	\begin{definition}
		We say that a NOVM $\Theta$ has bounded trace if and only if for any bounded Borel set $B\subseteq\mathbb{R}$, for any orthonormal basis of $\mathcal{H}$, $\left(e_n\right)_{n=1}^{\infty}$
		\begin{center}
			$\tr\Theta\left(B\right)\coloneqq\sum\limits_{n=1}^{\infty}\langle\Theta\left(B\right)e_j,e_j\rangle<\infty$.
		\end{center}
	\end{definition}
	\begin{remark}
		The trace of a nonnegative operator does not depend on the choice of the orthonormal basis.
	\end{remark}
	Throughout this work we will discuss certain limits which are called Riemann-Stieltjes integrals. These are integrals of operator-valued functions w.r.t.\ either PVMs or standard Borel measures. We present the case of a PVM here. Let us fix a $PVM$ $P$ and a $\mathcal{B}\left(\mathcal{H}\right)$-valued function \mbox{$F:\left[a,b\right]\to\mathcal{B}\left(\mathcal{H}\right)$}. Given a partition  of $\left[a,b\right]$ $\Delta=\left(x_0=a,x_1,\ldots,x_n=b\right)$ and a selection of points $\Delta^*=\left(t_1,\ldots,t_n\right)$ such that $t_i\in\left[x_{i-1},x_i\right)$, The Riemann Stieltjes sum of $F$ w.r.t.\ $\Delta$ and $\Delta^*$ is defined by
	\begin{center}
		$RS\left(F,\Delta,\Delta^*\right)\coloneqq\sum\limits_{k=1}^{n}F\left(t_j\right)P\left(\left[x_{j-1},x_j\right)\right)$.
	\end{center}
	We will say that $S\in\mathcal{B}\left(\mathcal{H}\right)$ is the Riemann-Stieltjes integral of $F$ w.r.t.\ $P$ if the Riemann-Stieltjes sums converge (in operator norm) to $S$ as $\left|\Delta\right|\to0$, where \mbox{$\left|\Delta\right|=\underset{i=1,\ldots,n}\max\left|x_i-x_{i-1}\right|$}. If this is the case, we also say that $F$ is integrable on $\left[a,b\right]$ and we denote $S=\int_a^b F\left(t\right)dP\left(t\right)$.
	\begin{remark}
		Given $a,b\in\mathbb{R}$, if $F$ is integrable on any interval $\left[a',b'\right]\subset\left(a,b\right)$, then we define
		\begin{center}
			$\int_{a}^{b}F\left(t\right)dP\left(t\right)\coloneqq\underset{a'\to a,b'\to b}{\slim}\int_{a'}^{b'}F\left(t\right)dP\left(t\right)$
		\end{center}
		if the limit on the RHS exists.
	\end{remark}
	The following is the well known spectral theorem (see, e.g.\ \cite[Theorem VII.8]{RS}).
	\begin{theorem}\label{spec_thm}
		Let $T:\mathcal{H}\to\mathcal{H}$ be a self-adjoint operator. Then there exists a projection-valued measure, $P_T$, such that for any continuous function $f:\sigma\left(T\right)\to\mathbb{C}$,
		\begin{equation}\label{spec_thm_eq}
			\int_{\sigma\left(T\right)}f\left(\lambda\right)dP_T\left(\lambda\right)=f\left(T\right),
		\end{equation}
		where the RHS is taken to be the limit, in operator norm, of the Riemann-Stieltjes sums. $P_T$ is called the spectral measure of $T$.
	\end{theorem}
	Given $\varphi,\psi\in\mathcal{H}$, their joint spectral measure w.r.t.\ $T$ is a Borel measure which is given by
	\begin{center}
		$\mu_{\varphi,\psi}\left(B\right)=\langle P_T\left(B\right)\varphi,\psi\rangle$
	\end{center}
	for any Borel set $B$. $\mu_{\varphi,\psi}^T$. We will write $\mu_\varphi^T$ when $\varphi=\psi$. and  we will omit the $T$ when it is clear from context. Let us now define the singular and absolutely continuous parts of a self-adjoint operator and of its spectral measure. To that end, let
	\begin{center}
		$\mathcal{H}_{\text{s}}=\left\{u\in\mathcal{H}:\mu_u\text{ is singular w.r.t.\ the Lebesgue measure}\right\}$,\\
		$\mathcal{H}_{\text{ac}}=\left\{u\in\mathcal{H}:\mu_u\text{ is absolutely continuous w.r.t.\ the Lebesgue measure}\right\}$.
	\end{center}
	It is well known (see, e.g. \cite[Theorem VII.4]{RS}) that $\mathcal{H}=\mathcal{H}_{\text{s}}\oplus\mathcal{H}_{\text{ac}}$, and that each of these subspaces is invariant under $T$. We will denote by $T_\bullet$ the operator $T|_{\mathcal{H}_\bullet}$ and by $\left(P_T\right)_\bullet$ its spectral measure, where $\bullet\in\left\{\text{s},\text{ac}\right\}$. We will omit the $T$ when it is clear from context. We will also be interested in scalar spectral measures of $T$, which are measures that are in the {\it spectral measure class} of $T$.
	\begin{definition}\label{measure_class_def}
		Let $\Theta$ be a NOVM and let $\rho$ be a Borel measure. We will say that $\rho$ is in the {\it measure class} of $\Theta$ if for any Borel set $B$, $\rho\left(B\right)=0\iff \Theta\left(B\right)=0$. Given a self-adjoint operator $T$, we will say that $\rho$ is in the {\it spectral measure class} of $T$ if it is in the measure class of $P_T$.
	\end{definition}
	We will also need the notion of a set which {\it supports} a measure. We define it here for scalar Borel measures. The definition for NOVMs is similar.
	\begin{definition}
		Let $\mu$ be a Borel measure. We will say that a Borel set $B\subseteq\mathbb{R}$ {\it supports} $\mu$ (or that $\mu$ is {\it supported} by $B$) if $\mu\left(\mathbb{R}\setminus B\right)=0$.
	\end{definition}
	\begin{lemma}\label{singular_part_lemma}
		Let $T:\mathcal{H}\to\mathcal{H}$ be a self-adjoint operator and let $\mu$ be a Borel measure which is in the spectral measure class of $T$. Let $A\subseteq\mathbb{R}$ be a Borel set of Lebesgue measure zero which supports $\mu_s$. Then $T_s=TP\left(A\right)$.
	\end{lemma}
	\begin{proof}
		This follows immediately from the fact that for any $\psi\in\mathcal{H}_{\text{ac}}$,
		\begin{center}
			$\|P\left(A\right)\psi\|^2=\langle P\left(A\right)\psi,P\left(A\right)\psi\rangle=\langle P\left(A\right)\psi,\psi\rangle=\int_A xd\mu_\psi=0$
		\end{center}
	\end{proof}
	\subsection{Subordinacy theory and the Borel transform}\label{sub_theory_borel_transform_section}
	\subsubsection{The half-line case}
	Let $J^+$ be a bounded self-adjoint Jacobi operator on $\ell^2\left(\mathbb{N}\right)$, namely an operator of the form
	\begin{equation}\label{jac_op_half_line}
		\left(J^+u\right)\left(n\right)=\begin{cases}
			a_{n-1}u\left(n-1\right)+a_nu\left(n+1\right)+b_nu\left(n\right) & n>1\\
			a_1u\left(2\right)+b_1u\left(1\right) & n=1
		\end{cases}
	\end{equation}
	for bounded real-valued sequences $\left(a_n\right)_{n\in\mathbb{N}}$ and $\left(b_n\right)_{n\in\mathbb{N}}$ such that $\left(a_n\right)_n\in\mathbb{N}$ is positive. For any $\theta\in\left[0,\pi\right)$, define the operator $J_\theta^+$ by\footnote{The case $\theta=\frac{\pi}{2}$ is defined by plugging in the sequences $\widetilde{a}_n=a_{n+1}$, $\widetilde{b}_n=b_{n+1}$ to (\ref{jac_op_half_line}).}
	\begin{center}
		$J_\theta^+\coloneqq J^+-\tan\theta\langle\delta_1,\cdot\rangle\delta_1$.
	\end{center}
	It is not hard to verify that for any $\theta\in\left[0,\pi\right)$, $\delta_1$ is a cyclic vector for $J^+_\theta$, namely \mbox{$\ell^2\left(\mathbb{N}\right)=\overline{\vspan\left\{\delta_1,J_\theta^+\delta_1,\left(J_\theta^+\right)^2\delta_1,\ldots\right\}}$}. In that case, the spectral measure of $\delta_1$, which we will denote by $\mu_\theta$, is in the spectral measure class of $J_\theta^+$. The purpose of subordinacy theory is to find a set $S$ which supports the singular part of $\mu_\theta$, and is defined through asymptotic properties of solutions of the eigenvalue equation,
	\begin{equation}\label{ev_eq1}
		a_{n-1}u\left(n-1\right)+a_n u\left(n+1\right)+b_nu\left(n\right)=Eu\left(n\right)
	\end{equation}
	subject to the boundary condition
	\begin{equation}\label{bc_eq_hl}
		u\left(0\right)\cos\theta+u\left(1\right)\sin\theta=0,
	\end{equation}
	by which we mean that $u$ is defined on $\mathbb{N}\cup\left\{0\right\}$, satisfies (\ref{ev_eq1}) for $n\in\mathbb{N}$, and satisfies (\ref{bc_eq_hl}).
	
	Given $u:\mathbb{N}\to\mathbb{C}$ and $L>1$, we define
	\begin{center}
		$\|u\|_L\coloneqq\left[\sum\limits_{n=1}^{\left[L\right]}\left|u\left(n\right)\right|^2+\left(L-\left[L\right]\left|u\left(\left[L\right]+1\right)\right|^2\right)\right]^\frac{1}{2}$.
	\end{center}
	In other words, $\|\cdot\|_L$ denotes the norm of a vector $u$ over an interval of length $L$. Given $E\in\mathbb{R}$ and $\theta\in\left[0,\pi\right)$, we will denote by $u_{\theta,E}$ the solution of (\ref{ev_eq1}) which satisfies (\ref{bc_eq_hl}), normalized so that $u_{\theta,E}\left(1\right)=\cos\theta$. With that in mind, we can present the definition of subordinacy.
	\begin{definition}
		Given $E\in\mathbb{R}$ and $\theta\in\left[0,\pi\right)$, $u_{\theta,E}$ will be called {\it subordinate} if for any $\eta\neq\theta$, we have
		\begin{equation}\label{subordinacy_eq}
			\underset{L\to\infty}{\lim}\frac{\|u_{\theta,E}\|_L}{\|u_{\eta,E}\|_L}=0.
		\end{equation}
	\end{definition}
	\begin{remark}
		For any $E\in\mathbb{R}$, the collection $\left\{\alpha u_{\theta,E}:\alpha\in\mathbb{C},\theta\in\left[0,\pi\right)\right\}$ is a two-dimensional vector space. This implies that $u_{\theta,E}$ is subordinate if and only if {\it there exists} $\eta\neq\theta$ such that (\ref{subordinacy_eq}) holds.
	\end{remark}
	We have the following
	\begin{theorem}\emph{(\cite{KP})}\label{sub_thm_hl}
		The part of $\mu_\theta$ which is singular w.r.t.\ the Lebesgue measure is supported on the set
		\begin{center}
			$S_\theta\coloneqq\left\{E\in\mathbb{R}:u_{\theta,E}\text{ is subordinate}\right\}$.
		\end{center}
	\end{theorem}
	\begin{remark}
		Originally, subordinacy theory was proved in \cite{GP} for continuum Schr\"{o}dinger operators which act on a subset of $L^2\left(\mathbb{R}_{\geq 0}\right)$. We state here the discrete setting which was proved in \cite{KP} as it is more relevant for this work.
	\end{remark}
	\subsubsection{The line case}
	Subordinacy theory was extended  in several ways (see, e.g.\ \cite{DGO,G,JL,L}). In this work, we will use its extension to Jacobi operators on $\ell^2\left(\mathbb{Z}\right)$. To that end, let $J:\ell^2\left(\mathbb{Z}\right)\to\ell^2\left(\mathbb{Z}\right)$ be given by (\ref{jac_op_line_eq}). $\delta_1$ is no longer necessarily a cyclic vector, but now the pair $\left\{\delta_0,\delta_1\right\}$ is cyclic, i.e.\ \mbox{$\ell^2\left(\mathbb{Z}\right)=\overline{\vspan\left\{J^n\delta_j:n\in\mathbb{N}\cup\left\{0\right\},j\in\left\{0,1\right\}\right\}}$}. This implies that $\mu\coloneqq\mu_{\delta_0}+\mu_{\delta_1}$ is in the spectral measure class of $J$. In the case of $\mathbb{Z}$, a solution $u$ of the equation
	\begin{equation}\label{ev_eq_line}
		u\left(n-1\right)+u\left(n+1\right)+V\left(n\right)u\left(n\right)=Eu\left(n\right)
	\end{equation}
	for some $E\in\mathbb{R}$ will be called subordinate if and only if its restrictions (under an appropriate normalization) to $\mathbb{Z}_+\coloneqq\left\{1,2,\ldots\right\}$ and to $\mathbb{Z}_-\coloneqq\left\{0,-1,\ldots\right\}$ are subordinate. We have the following
	\begin{theorem}\label{sub_thm_line}
		The singular part of $\mu$ is supported on the set
		\begin{equation}\label{sub_thm_line_supp_eq}
			S\coloneqq\left\{E\in\mathbb{R}:\text{(\ref{ev_eq_line}) has a subordinate solution}\right\}.
		\end{equation}
	\end{theorem}
	This theorem was proved for the continuum case in \cite{G}. A proof for the discrete case can be found in \cite{DF,L}.
	\subsubsection{The Borel transform of a measure}
	Given a Borel measure $\sigma$, its Borel transform is defined by
	\begin{center}
		$M_\sigma\left(z\right)=\int_\mathbb{R}\frac{d\sigma\left(E\right)}{E-z}$
	\end{center}
	for any $z\in\mathbb{C}$ such that $\im z>0$. The boundary behavior of the Borel transform of a measure $\mu$ is related to continuity properties of $\mu$ by the following theorem (see, for example, \cite[Theorem 1.6]{Sim}).
	\begin{theorem}\label{supp_thm}
		Let $\mu$ be a positive measure satisfying $\int_\mathbb{R}\frac{d\mu\left(x\right)}{\left|x\right|+1}<\infty$. Denote by $\mu_{ac},\mu_s$ the absolutely continuous and singular parts of $\mu$ (with respect to the Lebesgue measure) respectively. Then
		\begin{enumerate}
			\item $\mu_{ac}$ is supported on the set $\left\{E\in\mathbb{R}:0< \underset{\epsilon\to0}{\lim}\im m_\mu(E+i\epsilon)<\infty\right\}$.
		
			\item $\mu_s$ is supported on the set $\left\{E\in\mathbb{R}:\underset{\epsilon\to0}{\lim}\left|m_\mu(E+i\epsilon)\right|=\infty\right\}$.
		\end{enumerate}
	\end{theorem}
	\begin{remark}\label{sub_thm_equiv_formulation}
		The spectral measure of any vector w.r.t.\ any bounded self-adjoint operator satisfies the condition required in Theorem \ref{supp_thm}. With that in mind, there is a useful reformulation of Theorem \ref{sub_thm_hl}: Denote by $m_\theta$ the Borel transform of $\mu_\theta$. Then, using the formula
		\begin{equation}\label{pert_formula_m_func}
			m_\theta=\frac{m_0}{1-\tan\theta m_0},
		\end{equation}
		and an inequality which relates the growth rate of solutions to the behavior of the Borel transoform of $m_0$ (see \cite[Theorem 1.1]{JL}), one can show that $\underset{\epsilon\to0}{\lim}\left|m_\theta\left(E+i\epsilon\right)\right|=\infty$ if and only if $u_{\theta,E}$ is subordinate. In other words, we have
		\begin{equation}\label{JL_hl_sub_thm}
			u_{\theta,E}\text{ is subordinate}\iff\underset{\epsilon\to0}{\lim}m_0\left(E+i\epsilon\right)=\cot\theta.
		\end{equation}
	\end{remark}
	We will use the following theorem of Poltoratskii:
	\begin{theorem}\emph{(\cite[Theorem 1.1]{JakL})}\label{pol_thm}
		Let $\nu,\sigma$ be two Borel measures such that $\sigma$ is positive. Suppose that $\nu\ll\sigma$. Then for $\sigma_s$-almost every $E\in\mathbb{R}$, we have
		\begin{center}
			$\underset{\epsilon\to0}{\lim}\frac{M_\nu\left(E+i\epsilon\right)}{M_\sigma\left(E+i\epsilon\right)}=\frac{d\nu}{d\sigma}\left(E\right)$,
		\end{center}
		where $\sigma_s$ denotes the part of $\sigma$ which is singular w.r.t.\ the Lebesgue measure.
	\end{theorem}
	\begin{remark}
		This theorem was originally proved in \cite{Pol} for measures on the unit circle. Although the transition between the two versions is elementary, we cite here the real-valued version, which was proved in \cite{JakL}.
	\end{remark}
	We will also use the following
	\begin{lemma}\label{bdd_der_lemma}
		Let $T$ be a bounded self-adjoint operator on a Hilbert space $\mathcal{H}$ and let $\varphi_1,\varphi_2\in\mathcal{H}$. Let $\mu=\mu_{\varphi_1}+\mu_{\varphi_2}$. Then for any $k,j\in\left\{1,2\right\}$, $\mu_{\varphi_k,\varphi_j}\ll\mu$, and for $\mu$-almost every $E\in\mathbb{R}$, \mbox{$\left|\frac{d\mu_{\varphi_k\varphi_j}}{d\mu}\left(E\right)\right|\leq 1$}.
	\end{lemma}
	\begin{proof}
		For any Borel set $B\subseteq\mathbb{R}$ and $\psi,\varphi\in\mathcal{H}$, we have that $\mu_{\psi,\varphi}\left(B\right)=\langle P_T\left(B\right)\psi,\varphi\rangle$. This implies that 
		\begin{center}
			$\mu_{\varphi_k,\varphi_j}^2\left(B\right)=\left|\langle P_T\left(B\right)\varphi_k,\varphi_j\rangle\right|^2\leq\langle P_T\left(B\right)\varphi_k,\varphi_k\rangle\cdot\langle P_T\left(B\right)\varphi_j,\varphi_j\rangle\leq\mu^2\left(B\right)$,
		\end{center}
		which implies both that $\mu_{\varphi_k,\varphi_j}\ll\mu$ and that $\left|\frac{d\mu_{\varphi_k\varphi_j}}{d\mu}\right|\leq 1$ $\mu$-almost everywhere, as required.
	\end{proof}
	It will be useful to show that the subordinate solution is actually given in terms of the Radon-Nikodym derivative of certain spectral measures.  More precisely, we will use the following
	\begin{prop}\label{L_prop}
		Let $J$ be a Jacobi operator on $\ell^2\left(\mathbb{Z}\right)$, let $\mu\coloneqq\mu_{\delta_0}+\mu_{\delta_1}$ and denote by $M$ the Borel transform of $\mu$. For any $k,j\in\mathbb{Z}$, let $\mu_{kj}$ be the joint spectral measure of $\delta_k$ and $\delta_j$ w.r.t.\ $J$ and let $M_{kj}$ be its Borel transform.
		There exists a set $\widetilde{S}\subseteq S$ (where $S$ is given by (\ref{sub_thm_line_supp_eq})) which supports $\mu_s$ such that for any $E\in\widetilde{S}$ and for any $k,j\in\mathbb{Z}$, we have
		\begin{equation}
			\underset{\epsilon\to0}{\lim}
			\frac{\mu_{kj}\left(\left(E-\epsilon,E+\epsilon\right)\right)}{\mu\left(E-\epsilon,E+\epsilon\right)}=\underset{\epsilon\to 0}{\lim}\frac{M_{kj}\left(E+i\epsilon\right)}{M\left(E+i\epsilon\right)}=u_E\left(k\right)u_E\left(j\right),
		\end{equation}
		where $u_E$ is the subordinate solution of (\ref{ev_eq_line}) which satisfies $\left|u_E\left(0\right)\right|^2+\left|u_E\left(1\right)\right|^2=1$.
	\end{prop}
	The proof of Proposition \ref{L_prop} mostly follows the proof of \cite[Theorem 1.4]{L}, and so it is given in the Appendix.
	\section{Eigenfunction expansion for Jacobi operators on $\ell^2\left(\mathbb{Z}\right)$}
	Let $J$ be a Jacobi operator on $\ell^2\left(\mathbb{Z}\right)$. Our goal in this section is to obtain an eigenfunction expansion of the singular part of $J$ following the lines of Berezanski\u{i} \cite{Ber} with slight modifications. Throughout this section, we denote the spectral measure of $J$ by $P$. We will also denote $\mu=\mu_{\delta_0}+\mu_{\delta_1}$, and $S$ the support of $\mu$ which is given by Theorem \ref{sub_thm_line}. We will use the following theorem:
	\begin{prop}\emph{(\cite[Chapter V, Theorem 1.1]{Ber})}\label{bdd_trace_eig_exp_thm}
		Let $\Theta$ be a NOVM which has bounded trace, and  let $\rho$ be a positive Borel measure which is in the measure class of $\Theta$. Then there exists an operator-valued function $\Psi:\mathbb{R}\to\mathcal{B}\left(\mathcal{H}\right)$ which is defined $\rho$-almost everywhere, such that for any Borel set $B\subseteq\mathbb{R}$,
		\begin{equation}\label{bdd_tr_int_eq}
			\Theta\left(B\right)=\int_B\Psi\left(E\right)d\rho\left(E\right),
		\end{equation}
		where the Riemann-Stieltjes integral in (\ref{bdd_tr_int_eq}) converges in operator norm. For $\rho$-almost every $E\in\mathbb{R}$, The operator $\Psi$ can be taken to be the weak limit of $\frac{\Theta\left(I_{E,n}\right)}{\rho\left(I_{E,n}\right)}$, where $I_{E,n}=\left(E-\frac{1}{n},E+\frac{1}{n}\right)$. In particular, for $\rho$-almost every $E\in\mathbb{R}$, the weak limit exists.
	\end{prop}
	\begin{remark}
		In \cite{Ber}, the theorem is proven for $\rho=\tr\Theta$. The proof remains exactly the same when one takes $\rho$ to be an arbitrary measure which is in the measure class of $\Theta$.
	\end{remark}
	Our goal is to obtain a formula similar to (\ref{bdd_tr_int_eq}) with $P_s$ instead of $\Theta$, where $P_s$ is the singular part of the spectral measure of $J$. To that end, let us introduce the operator $A:D\left(A\right)\subset\ell^2\left(\mathbb{Z}\right)\to\ell^2\left(\mathbb{Z}\right)$ which is given by
	\begin{center}
		$\left(A\psi\right)\left(n\right)=n\psi\left(n\right)$
	\end{center}
	and is defined on the maximal domain
	\begin{center}
		$D\left(A\right)\coloneqq\left\{\psi\in\ell^2\left(\mathbb{Z}\right):A\psi\in\ell^2\left(\mathbb{Z}\right)\right\}$.
	\end{center}
	Note that $D\left(A\right)$ contains all of the sequences which are compactly supported. Also note that $A$ has a bounded and self-adjoint inverse.
	\begin{claim}\label{NOVM_claim}
		For any Borel set $B\subseteq\mathbb{R}$, let $\Theta\left(B\right)\coloneqq A^{-1}P\left(B\right)A^{-1}$. Then $\Theta$ is a NOVM with a bounded trace.
	\end{claim}
	\begin{proof}
		The fact that $P$ is a PVM along with $A^{-1}$ being a positive operator imply that $\Theta$ is a NOVM. We turn to show that $\Theta$ has bounded trace. Let $B\subseteq\mathbb{R}$ be a Borel set.
		\begin{equation*}
			\begin{split}
			\sum\limits_{n\in\mathbb{Z}}\langle\Theta\left(B\right)\delta_n,\delta_n\rangle=\sum\limits_{n\in\mathbb{Z}}\langle A^{-1}P\left(B\right)A^{-1}\delta_n,\delta_n\rangle=\sum\limits_{n\in\mathbb{Z}}\frac{1}{n^2}\langle P\left(B\right)\delta_n,\delta_n\rangle\leq\sum\limits_{n\in\mathbb{Z}}\frac{1}{n^2}<\infty,
			\end{split}
		\end{equation*}
		where the last equality follows from the fact that $P\leq\text{Id}$.
	\end{proof}
	By Claim \ref{NOVM_claim}, $\Theta$ satisfies the conditions of Proposition \ref{bdd_trace_eig_exp_thm}. In addition, $\mu$ is in the measure class of $\Theta$ and so (\ref{bdd_tr_int_eq}) (with $\rho=\mu$) holds for some operator-valued function $\Psi$ which is defined $\mu$-almost everywhere. Furthermore, $\Psi$ is given by
	\begin{center}
		$\Psi=\underset{n\to\infty}{\wlim}\frac{\Theta\left(I_{E,n}\right)}{\mu\left(I_{E,n}\right)}$,
	\end{center}
	and so for $\mu$-almost every $E\in\mathbb{R}$, we have
	\begin{equation}\label{mat_rep_psi_eq}
		\langle\Psi\left(E\right)\delta_k,\delta_j\rangle=\underset{n\to\infty}{\lim}\frac{1}{\mu\left(I_{E,n}\right)}\langle T^{-1}\Theta\left(I_{E,n}\right)T^{-1}\delta_k,\delta_j\rangle=\underset{n\to\infty}{\lim}\frac{1}{k^2}\frac{\mu_{kj}\left(I_{E,n}\right)}{\mu\left(I_{E,n}\right)}=\frac{1}{k^2}\frac{d\mu_{kj}}{d\mu}\left(E\right).
	\end{equation}
	By Theorem \ref{pol_thm} and Proposition \ref{L_prop}, there exists a Borel set $\widetilde{S}\subseteq S$ such that $\mu_s\left(S\setminus\widetilde{S}\right)=0$, and for any $E\in\widetilde{S}$, $k,j\in\mathbb{Z}$, we have
	\begin{equation}\label{S_tilde_eq}
		\underset{n\to\infty}{\lim}\frac{\mu_{kj}\left(I_{E,n}\right)}{\mu\left(I_{E,n}\right)}=\underset{\epsilon\to0}{\lim}\frac{M_{kj}\left(E+i\epsilon\right)}{M\left(E+i\epsilon\right)}=u_E\left(k\right)u_E\left(j\right),
	\end{equation} where $u_E$ is the subordinate solution of (\ref{ev_eq_line}) which satisfies $\left|u_E\left(0\right)\right|^2+\left|u_E\left(1\right)\right|^2=1$. We can now prove Theorem \ref{main_thm_1_intro}.
\begin{proof}
	Define $\mathcal{H}_\pm\coloneqq\ell^2\left(\mathbb{Z};\left|n\right|^{\pm 2}\right)$. For any $\varphi\in\mathcal{H}$, $A\varphi\in\mathcal{H}_-$ and in addition, $\mathcal{H}_+=D\left(A\right)$. Thus, for any $E\in\widetilde{S}$, the operator $\Phi\left(E\right)\coloneqq A\Psi\left(E\right)A$ is defined from $\mathcal{H}_+$ to $\mathcal{H}_-$. Furthermore, for any $u,v\in\mathcal{H}_+$, we have
	\begin{center}
		$\langle\Psi\left(E\right)Au,Av\rangle=\langle A\Psi\left(E\right)Au,v\rangle=\langle \Phi\left(E\right)u,v\rangle$
	\end{center}
	and so for any Borel set $B\subseteq\mathbb{R}$,
	\begin{equation}\label{weak_conv_eq}
		\begin{split}
			\langle P\left(B\right)u,v\rangle=\langle P\left(B\right)A^{-1}Au,A^{-1}Av\rangle=\langle A^{-1}P\left(B\right)A^{-1}Au,Av\rangle=\\=\langle\Theta\left(B\right)Au,Av\rangle=\int_B\langle\Psi\left(E\right)Au,Av\rangle d\mu\left(E\right)=\int_B\langle\Phi\left(E\right)u,v\rangle d\mu\left(E\right).
		\end{split}
	\end{equation}
	Thus, the integral
	\begin{equation}\label{weak_conv_int_eq}
		P\left(B\right)=\int_B \Phi\left(E\right)d\mu\left(E\right)
	\end{equation}
	converges weakly. Note that $\Phi\left(E\right)\in\mathcal{B}\left(\mathcal{H}_+,\mathcal{H}_-\right)$ and in addition, \mbox{$\|\Phi\left(E\right)\|\leq 1$}. Thus, the weak convergence of (\ref{weak_conv_int_eq}) implies convergence in the operator norm of $\mathcal{B}\left(\mathcal{H}_+,\mathcal{H}_-\right)$. Furthermore, by (\ref{mat_rep_psi_eq}) and (\ref{S_tilde_eq}), for $\mu_s$-almost every $E\in\mathbb{R}$, $\Phi\left(E\right)=\langle u_E,\cdot\rangle u_E$, and so for any Borel set $B$ and $u\in\mathcal{H}_+$, we get
	\begin{center}
		$P_s\left(B\right)u=\int_{B\cap S}\Phi\left(E\right)ud\mu_s\left(E\right)=\int_{B\cap S}\langle u_E,u\rangle u_Ed\mu_s\left(E\right)$,
	\end{center}
	as required.
\end{proof}
	\section{$J_s$ as an integral of half-line operators}\label{integral_sec}
	The notation used in the prior section applies throughout this section as well. We will also denote $I\coloneqq\left[0,\pi\right)$.
	\subsection{The correlative spectral resolution}
	In this subsection, we will define a projection-valued measure $Q$ on $\left[0,\pi\right)$ which, in some sense, measures the correlation between the restrictions of $J$ to the half-lines $\mathbb{Z}_\pm$. Recall that $S\subseteq\mathbb{R}$ is the set of all $E\in\mathbb{R}$ for which there exists a subordinate solution of (\ref{ev_eq_line}), and that $P_s$ is supported on the set $\widetilde{S}\subseteq S$ which was defined in the previous section. Denote by $J_{\pm}$ the operators
	\begin{center}
		$J_{\pm}=P_{\mathbb{Z}_{\pm}}JP_{\mathbb{Z}_{\pm}}$
	\end{center}
	where $\mathbb{Z}_+=\left\{1,2,\ldots\right\}$ and $\mathbb{Z}_-=\mathbb{Z}_+^c$. For any $\theta\in\left[0,\pi\right)$, we define $J_+^\theta=J_+-\tan\theta\langle\delta_1,\cdot\rangle\delta_1$, and $J_-^{\theta}=J_--\cot\theta\langle\delta_0,\cdot\rangle\delta_0$.
	Recall that for any $E\in\mathbb{R}$ and $\theta\in\left[0,\pi\right)$, $u_{\theta,E}$ is the solution of (\ref{ev_eq1}) which satisfies (\ref{bc_eq_hl}), and let 
	\begin{center}
		$S_{\theta}^{\pm}=\left\{E\in\mathbb{R}:u_{\theta,E}^\pm\text{ is subordinate}\right\}$.
	\end{center}
	Let us denote $S_\theta\coloneqq S_-^\theta\cap S_+^\theta$. It is not hard to see that
	$S=\underset{\theta\in I}{\bigcup}S_\theta$. In addition, by (\ref{S_tilde_eq}) and (\ref{bc_eq_hl}), for any $\theta\in\left[0,\pi\right)$ and for any $E\in\widetilde{S}$,
	\begin{equation}\label{measurability_eq}
		E\in S_\theta\iff\underset{\epsilon\to0}{\lim}\frac{M_0\left(E+i\epsilon\right)}{M\left(E+i\epsilon\right)}=\frac{d\mu_{00}}{d\mu}\left(E\right)=\cos\theta.
	\end{equation}
	\begin{prop}
		For any Borel set $B\subseteq I$, the set $S_B\coloneqq\underset{\theta\in B}{\bigcup} S_\theta\cap\widetilde{S}$ is Borel measurable.
	\end{prop}
	\begin{proof}
		Denote $U=\cos\left(B\right)$ and let $f=\frac{d\mu_{00}}{d\mu}\left(E\right)$ given by (\ref{measurability_eq}). Then $S_B=f^{-1}\left(U\right)$, which is a measurable set.
	\end{proof}
	Now, for any Borel set $B\subseteq I$, we can define
	\begin{center}
		Q$\left(B\right)=P\left(S_B\right)$.
	\end{center}
	\begin{claim}
		$Q$ is a projection-valued measure which commutes with $J$.
	\end{claim}
	\begin{proof}
		This is immediate by the fact that $P$ is a projection-valued measure which commutes with $J$.
	\end{proof}
	One of the main reasons to consider the projection-valued measure $Q$ is that it enables one to split $\ell^2\left(\mathbb{Z}\right)$ into a direct sum of subspaces such that on each subspace, $J$ looks like the direct sum of half-line operators. For any $\theta\in\left[0,\pi\right)$, we define $J_\theta\coloneqq J^\theta_-\oplus J^\theta_+$.
	\begin{theorem}\label{splitting_thm}
		Let $a,b\in\left(0,\frac{\pi}{2}\right)$ and let $\mathcal{F}=\vspan\left(\mathcal{F}_1\cup\mathcal{F}_2\cup\mathcal{F}_2\right)\subseteq C\left(\mathbb{R}\right)$, where
		\begin{center}
			$\mathcal{F}_1=\left\{x^n:n\in\mathbb{N}\right\}$,\\
			$\mathcal{F}_2=\left\{e^{itx}:t\in\mathbb{R}\right\}$
			\\
			$\mathcal{F}_3=\left\{\left(x-z\right)^{-1}:z\in\mathbb{C},\im z>0\right\}$.
		\end{center}
		For any $F\in\mathcal{F}$ there exists $\gamma>0$ such that for any $a\leq\alpha<\beta\leq b$, $\psi\in\mathcal{H}_+$ and $\theta\in\left[\alpha,\beta\right)$,
		\begin{equation}\label{decomp_ineq}
			\|F\left(J_\theta\right)Q\left(\left[\alpha,\beta\right)\right)\psi-F\left(J
			\right)Q\left(\left[\alpha,\beta\right)\right)\psi\|<\gamma\left|\beta-\alpha\right|\|Q\left(\left[\alpha,\beta\right)\right)\psi\|.
		\end{equation}
	\end{theorem}
	\begin{proof}
		We will prove the theorem for basis elements. The general case then easily follows using the triangle inequality. As before, for any $E\in S$, we denote by $u_E$ the subordinate solution of (\ref{ev_eq_line}) which satisfies (\ref{bc_eq_hl}). Let $0<M\in\mathbb{R}$ be some constant such that $\|J\|\leq M$ and for any $\theta\in\left[a,b\right]$, $\|J_\theta\|\leq M$. We treat each $\mathcal{F}_i$ separately.\\\textbf{First case:} For $F\in\mathcal{F}_1$, we will prove the theorem by induction on $n$. Suppose $F\in\mathcal{F}$ is given by $F\left(x\right)=x$. A straightforward computation shows that for any $\psi\in\ell^2\left(\mathbb{Z}\right)$,
		\begin{equation}\label{diff_J_Jtheta}
			\left(J_\theta-J\right)\psi=\left(\psi\left(1\right)+\cot\theta\psi\left(0\right)\right)\delta_0+\left(\psi\left(0\right)+\tan\theta\psi\left(1\right)\right)\delta_1.
		\end{equation}
		In addition, let $A=\underset{\theta\in \left[\alpha,\beta\right)}{\bigcup} S_\theta$. By Theorem \ref{main_thm_1_intro}, for any $\psi\in\mathcal{H}_+$ and for any $n\in\mathbb{Z}$,
		\begin{equation}
			\left(Q\left(\left[\alpha,\beta\right)\right)\psi\right)\left(n\right)=\int_A\langle\psi,\psi_E\rangle\psi_E\left(n\right) d\mu\left(E\right)
		\end{equation}
		and in addition,
		\begin{center}
				$\|J_\theta Q\left(\left[\alpha,\beta\right)\right)\psi-JQ\left(\left[\alpha,\beta\right)\right)\psi\|^2=\|\left(J_\theta-J\right)Q\left(\left[\alpha,\beta\right)\right)\psi\|^2=$\\ \text{}\\$=\left|\left(J-J_\theta\right)Q\left(\left[\alpha,\beta\right)\right)\psi\left(0\right)\right|^2+\left|\left(J-J_\theta\right)Q\left(\left[\alpha,\beta\right)\right)\psi\left(1\right)\right|^2.$
		\end{center}
		where the last equality follows from (\ref{diff_J_Jtheta}). Recall that $S_{\left[\alpha,\beta\right]}=\underset{\theta\in\left[\alpha,\beta\right]}{S_\theta}$ and for any $E\in S_{\left[\alpha,\beta\right]}$, let $\theta\left(E\right)$ be the unique $\theta\in\left[\alpha,\beta\right)$ such that $E\in S_\theta$. Now, note that
		\begin{center}
			$\left|\left(J-J_\theta\right)Q\left(\left[\alpha,\beta\right]\right)\psi\left(1\right)\right|^2=\left|\int_{S_{\left[\alpha,\beta\right]}}\langle\psi,\psi_E\rangle\psi_E\left(0\right)d\mu\left(E\right)+\tan\theta\int_{S_{\left[\alpha,\beta\right]}}\langle\psi,\psi_E\rangle\psi_E\left(1\right)d\mu\left(E\right)\right|^2=$\\ \text{}\\$=\left|\int_{S_{\left[\alpha,\beta\right]}}\left(\tan\theta-\tan\theta\left(E\right)\right)\langle\psi,\psi_E\rangle\psi_E\left(1\right)d\mu\left(E\right)\right|^2\leq\underset{\eta\in\left[\alpha,\beta\right]}{\max}\left|\tan\theta-\tan\eta\right|^2\left|\int_{S_{\left[\alpha,\beta\right]}}\langle\psi,\psi_E\rangle\psi_E\left(1\right)d\mu\left(E\right)\right|^2=$\\ \text{}\\$=\underset{\eta\in\left[\alpha,\beta\right]}{\max}\left|\tan\theta-\tan\eta\right|^2\|Q\left(\left[\alpha,\beta\right]\right)\psi\|^2$.
		\end{center}
		In addition, $\tan$ satisfies the Lipschitz condition on $\left[a,b\right]$ and so there exists $\gamma>0$ such that  $\left|\tan\theta-\tan\eta\right|<\gamma\left|\theta-\eta\right|$. This immediately implies the required inequality. The analysis of the second summand is done in a similar way, noting that $\cot\theta$ also satisfies the Lipschitz condition on $\left[a,b\right]$. Now suppose that the theorem holds for $x^{n-1}$ and let $F\in\mathcal{F}$ be given by $F\left(x\right)=x^n$. We have
		\begin{center}
			$\|J_\theta^n Q\left(\left[\alpha,\beta\right)\right)\psi-J^nP\left(\left[\alpha,\beta\right)\right)\psi\|=\|J_\theta^{n-1}J_\theta Q\left(\left[\alpha,\beta\right)\right)\psi-J^{n-1}JQ\left(\left[\alpha,\beta\right)\right)\psi\|\leq$\\ \text{}\\$\leq\|J_\theta^{n-1}J_\theta^nQ\left(\left[\alpha,\beta\right)\right)\psi-J_\theta^{n-1}JQ\left(\left[\alpha,\beta\right)\right)\psi\|+\|J_\theta^{n-1}JQ\left(\left[\alpha,\beta\right)\right)\psi-J^{n-1}JQ\left(\left[\alpha,\beta\right)\right)\psi\|=$\\ \text{}\\$=\|J_\theta^{n-1}\big[J_\theta Q\left(\left[\alpha,\beta\right)\right)\psi-JQ\left(\left[\alpha,\beta\right)\right)\psi\big]\|+\|J_\theta^{n-1}Q\left[\alpha,\beta\right)J\psi-J^{n-1}Q\left[\alpha,\beta\right)J\psi\|\leq$\\ \text{}\\$\leq\|J_\theta^{n-1}\|\cdot\gamma_1\left|\beta-\alpha\right|\|Q\left(\left[\alpha,\beta\right)\psi\right)\|+\gamma_2\left|\beta-\alpha\right|\|Q\left(\left[\alpha,\beta\right)\right)J\psi\|$
		\end{center}
		where the last inequality follows from the induction hypothesis for suitable constants $\gamma_1,\gamma_2$. Note that $J$ and $Q$ commute, and so we may proceed
		\begin{center}
			$\leq M^{n-1}\cdot\gamma_1\left|\beta-\alpha\right|\|Q\left(\left[\alpha,\beta\right)\right)\psi\|+\gamma_2\left|\beta-\alpha\right|\|JQ\left(\left[\alpha,\beta\right)\right)\psi\|\leq$\\ \text{}\\$\leq\left(\underset{=\gamma}{\overset{\left(\star\right)}{\underbrace{M^{n-1}\gamma_1+M\gamma_2}}}\right)\left|\beta-\alpha\right|\|Q\left(\left[\alpha,\beta\right)\right)\psi\|$
		\end{center}
		as required. Note also that using ($\star$), if we denote by $\gamma_n$ the constant which corresponds with $x^n$, then it can easily be shown by induction that there exists $C>0$ such that $\gamma_n\leq CM^n$.\\\textbf{Second case:} Let $t\in\mathbb{R}$, $\psi\in\mathcal{H}_+$. We have
		\begin{center}
			$\|e^{itJ_\theta}Q\left(\left[\alpha,\beta\right)\right)\psi-e^{itJ}Q\left(\left[\alpha,\beta\right)\right)\psi\|=\|\sum\limits_{n=0}^{\infty}\frac{\left(itJ_\theta\right)^n}{n!}Q\left(\left[\alpha,\beta\right)\right)\psi-\frac{\left(itJ\right)^n}{n!}Q\left(\left[\alpha,\beta\right)\right)\psi\|\leq\sum\limits_{n=0}^\infty\frac{\left|t\right|^n\gamma_n}{n!}\left|\beta-\alpha\right|\|Q\left(\left[\alpha,\beta\right)\right)\psi\|<C\left|\beta-\alpha\right|\|Q\left(\left[\alpha,\beta\right)\right)\psi\|\sum\limits_{n=0}^\infty\frac{\left|tM\right|^n}{n!}=Ce^{tM}\left|\beta-\alpha\right|\|Q\left(\left[\alpha,\beta\right)\right)\psi\|$.
		\end{center}
		Setting $\gamma=Ce^{tM}$, we get the desired result.\\\textbf{Third case:} For functions in $\mathcal{F}_3$, we have the resolvent formula
		\begin{equation}\label{resolvent_formula}
			\left(T-z\right)^{-1}-\left(S-z\right)^{-1}=\left(T-z\right)^{-1}\left(S-T\right)\left(S-z\right)^{-1}
		\end{equation}
		for self-adjoint operators $T,S\in\mathcal{B}\left(\mathcal{H}\right)$ and $z\in\mathbb{C}\setminus\mathbb{R}$. In addition, $\mathcal{H}_+$ is invariant under $\left(J-z\right)^{-1}$. Thus, we have
		\begin{center}
			$\|\left(J_\theta-z\right)^{-1}Q\left(\left[\alpha,\beta\right)\right)\psi-\left(J-z\right)^{-1}Q\left(\left[\alpha,\beta\right)\right)\psi\|=$\\ \text{}\\$=\|\left(J_\theta-z\right)^{-1}\left(J-J_\theta\right)\left(J-z\right)^{-1}Q\left(\left[\alpha,\beta\right)\right)\psi\|=\|\left(J_\theta-z\right)^{-1}\left(J-J_\theta\right)Q\left(\left[\alpha,\beta\right)\right)\left(J-z\right)^{-1}\psi\|\leq\|\left(J_\theta-z\right)^{-1}\|\cdot\gamma\left|\beta-\alpha\right|\|Q\left(\left[\alpha,\beta\right)\right)\left(J-z\right)^{-1}\psi\|=\|\left(J_\theta-z\right)^{-1}\|\gamma\left|\beta-\alpha\right|\|\left(J-z\right)^{-1}Q\left(\left[\alpha,\beta\right)\right)\psi\|\leq$\\ \text{}\\$\leq M^2\gamma\left|\beta-\alpha\right|\|Q\left(\left[\alpha,\beta\right)\right)\psi\|$
		\end{center}
		which implies the desired result.
	\end{proof}
	\begin{remark}
		Note that Theorem \ref{splitting_thm} holds also when taking $a,b\in\left(\frac{\pi}{2},\pi\right)$.
	\end{remark}
	Theorem \ref{main_thm_2_intro} now follows immediately from Theorem \ref{splitting_thm} and Lemma \ref{singular_part_lemma}.
	\subsection{Integrating w.r.t.\ a projection-valued measure}
	Our goal in this section is to prove a (stronger version of) Theorem \ref{main_thm_2_intro}. We will use the following lemma:
	\begin{lemma} \label{int_lemma}
		Let $P$ be a PVM, $T\in\mathcal{B}\left(\mathcal{H}\right)$ self-adjoint and $F:\left[a,b\right]\to\mathcal{B}\left(\mathcal{H}\right)$. Suppose that there exists $M\in\mathbb{R}$ such that $\|T\|\leq M$ and for any $t\in\left[a,b\right]$, $\left\|F\left(t\right)\right\|\leq M$. In addition, suppose that there exists a dense subspace $\mathcal{H}_+\subseteq\mathcal{H}$ and $\gamma>0$ such that for any $a\leq\alpha<\beta\leq b$ and for any $\psi_0\in\mathcal{H}_+$,
		\begin{equation}\label{splitting_thm_cond}
			\|F\left(t\right)P\left(\left[\alpha,\beta\right)\right)\psi_0-TP\left(\left[\alpha,\beta\right)\right)\psi\|<\gamma\left|\beta-\alpha\right|\|P\left(\left[\alpha,\beta\right)\right)\psi_0\|.
		\end{equation}
		Then $F$ is integrable w.r.t.\ $P$ and 
		\begin{center}
			$\int_a^b F\left(t\right)dP\left(t\right)=TP\left(\left[a,b\right)\right)$.
		\end{center}
	\end{lemma}
	\begin{proof}
		We first claim that (\ref{splitting_thm_cond}) holds for any $\psi\in\mathcal{H}$. Indeed, given $\psi\in\mathcal{H}$ and $\psi_0\in\mathcal{H}_+$, we have
		\begin{center}
			$\|F\left(t\right)P\left(\left[\alpha,\beta\right)\right)\psi-TP\left(\left[\alpha,\beta\right)\right)\psi\|\leq\|F\left(t\right)P\left(\left[\alpha,\beta\right)\right)\psi-F\left(t\right)P\left(\left[\alpha,\beta\right)\right)\psi_0\|+\|F\left(t\right)P\left(\left[\alpha,\beta\right)\right)\psi_0-TP\left(\left[\alpha,\beta\right)\right)\psi_0\|+\|TP\left(\left[\alpha,\beta\right)\right)\psi_0-TP\left(\left[\alpha,\beta\right)\right)\psi\|<\|F\left(t\right)\|\|\psi-\psi_0\|+\gamma\left|\beta-\alpha\right|\|P\left(\left[\alpha,\beta\right)\right)\psi_0\|+\|T\|\|\psi-\psi_0\|\leq\gamma\left|\beta-\alpha\right|\|P\left(\left[\alpha,\beta\right)\right)\psi_0\|+\left(\|F\left(t\right)\|+\|T\|+\|P\left(\left[\alpha,\beta\right)\right)\|\right)\|\psi-\psi_0\|\leq\gamma\left|\beta-\alpha\right|\|P\left(\left[\alpha,\beta\right)\right)\psi_0\|+\left(2M+1\right)\|\psi-\psi_0\|$.
		\end{center}
		Letting $\psi_0\to\psi$, we get (\ref{splitting_thm_cond}). Now, let $\psi\in\mathcal{H}$, let $\Delta=\left(a=x_0,
		\ldots,x_n=b\right)$ be any partition and let $\left(t_j\right)_{j=1}^n$ be any choice of points.
		\begin{center}
			$\|\sum\limits_{j=1}^{n}F\left(t_j\right)P\left(\left[x_{j-1},x_j\right)\right)\psi-T\psi\|=\|\sum\limits_{j=1}^nF\left(t_j\right)P\left(\left[x_{j-1},x_j\right)\right)\psi-TP\left(\left[x_{j-1},x_j\right)\right)\psi\|<\sum\limits_{j=1}^n\gamma\left|x_j-x_{j-1}\right|\|P\left(\left[x_{j-1},x_j\right)\right)\psi\|\leq\gamma\left(\sum\limits_{j=1}^n\left|x_j-x_{j-1}\right|^2\right)^{\frac{1}{2}}\left(\sum\limits_{j=1}^n\|P\left(\left[x_{j-1},x_j\right)\right)\psi\|^2\right)^\frac{1}{2}\leq\gamma\left|\Delta\right|\sqrt{b-a}\|\psi\|\underset{\left|\Delta\right|\to0}{\longrightarrow}0$,
		\end{center}
		as required.
	\end{proof}
	Lemma \ref{int_lemma} and Theorem \ref{splitting_thm} imply
	\begin{theorem}\label{integral_thm}
		Let $\mathcal{F},\,Q$ be as in Theorem \ref{splitting_thm}. Then for any $F\in\mathcal{F}$,
		\begin{equation}\label{int_eq}
			\int_I F\left(J_\theta\right)dQ\left(\theta\right)=F\left(J\right)Q\left(I\right),
		\end{equation}
		where $I=\left[0,\pi\right)$.
	\end{theorem}
	\begin{proof}
		Let $a_n=\frac{1}{n}$, $b_n=\frac{\pi}{2}-\frac{1}{n}$. Then, by Theorem \ref{splitting_thm}, (\ref{splitting_thm_cond}) holds for $F\left(J_\theta\right)$ and $F\left(J\right)Q\left(\left[a_n,b_n\right]\right)$. Thus, by Lemma \ref{int_lemma} we get that
		\begin{center}
			$\int_{a_n}^{b_n}F\left(J_\theta\right)dQ\left(\theta\right)=F\left(J\right)Q\left(\left[a_n,b_n\right]\right)$.
		\end{center}
		Letting $n\to\infty$ and noting that $F\left(J\right)Q\left(\left[a_n,b_n\right]\right)\to F\left(J\right)Q\left(\left(0,\frac{\pi}{2}\right)\right)$ as $n\to\infty$, we get that
		\begin{center}
			$\int_{I_1} F\left(J_\theta\right)dQ\left(\theta\right)=F\left(J\right)Q\left(I_1\right)$,
		\end{center}
		where $I_1=\left(0,\frac{\pi}{2}\right)$. In a similar way, for $I_2=\left(\frac{\pi}{2},\pi\right)$,
		\begin{center}
			$\int_{I_2}F\left(J_\theta\right)dQ\left(\theta\right)=F\left(J\right)Q\left(I_2\right)$.
		\end{center}
		and so
		\begin{center}
			$\int_{I_1\cup I_2}F\left(J_\theta\right)dQ\left(\theta\right)=F\left(J\right)Q\left(I_1\cup I_2\right)$.
		\end{center}
		Now, note that for any $\theta\in\left[0,\pi\right)$ such that $Q\left(\left\{\theta\right\}\right)\neq0$, $J|_{\text{Ran}Q\left(\left\{\theta\right\}\right)}=J_\theta|_{\text{Ran}Q\left(\left\{\theta\right\}\right)}$ and so for any $F\in\mathcal{F}$, \mbox{$F\left(J\right)|_{\text{Ran}Q\left(\left\{\theta\right\}\right)}=F\left(J_\theta\right)|_{\text{Ran}Q\left(\left\{\theta\right\}\right)}$}. Thus, trivially,
		\begin{center}
			$\int_{\left\{\theta\right\}}F\left(J_\theta\right)dQ\left(\theta\right)=F\left(J\right)Q\left(\left\{\theta\right\}\right)$.
		\end{center}
		Thus, since $I=I_1\cup I_2\cup\left\{0\right\}\cup\left\{\frac{\pi}{2}\right\}$, we get (\ref{int_eq}), as required.
	\end{proof}
	Theorem \ref{main_thm_3_intro} now follows from Theorem \ref{integral_thm} and Lemma \ref{singular_part_lemma}, along with the fact that $Q\left(I\right)=P\left(S\right)$, and $S$ supports $\mu_s$.
	\section{Examples and further discussion}
	\subsection{An operator with a symmetric potential}\label{symm_potential_subsec}
	Let us discuss a special case of Jacobi operators in which the potential $b$ satisfies some symmetry condition. For simplicity, in this example we will assume that $a\equiv 1$. We say that $b$ is {\it symmetric around $\frac{1}{2}$} if for any $n\in\mathbb{Z}$, $b_n=b_{-n+1}$.
	\begin{prop}\label{symmetric_V_prop}
		Suppose $b$ is symmetric around $\frac{1}{2}$. Then $Q$ is given by $Q=Q_{\text{Sym}}+Q_{\text{Sym}^\perp}$, where
		\begin{center}
			$Q_{\text{Sym}}\left(A\right)=\begin{cases}
				P_{\text{Sym}} & -\frac{\pi}{4}\in A\\
				0 & \text{otherwise}
			\end{cases},\,\,\,\,
		Q_{\text{Sym}^\perp}\left(A\right)=\begin{cases}
			\text{Id}-P_{\text{Sym}} & \frac{\pi}{4}\in A\\
			0 & \text{otherwise}
		\end{cases}$
		\end{center}
		and $P_{\text{Sym}}$ is the orthogonal projection onto $\mathcal{H}_{\text{Sym}}\coloneqq\left\{\psi\in\mathcal{H}:\psi\left(n\right)=\psi\left(-n+1\right)\right\}$.
	\end{prop}
	\begin{proof}
		By the symmetry of the potential, we get that $J_+^0=J_-^{\frac{\pi}{2}}$. For any $\theta\in\left[0,\pi\right)$, Let $\mu_+^\theta$ ($\mu_-^\theta$) be the spectral measure of $\delta_1$ ($\delta_0$) w.r.t.\ $J_+^\theta$ ($J_-^\theta$). Also let $m_\pm^\theta$ be the Borel transform of $\mu_{\pm}^\theta$. We claim that for any $\theta\in\left[0,\pi\right)$, for any $E\in\mathbb{R}$, $\underset{\epsilon\to0}{\lim}\left|m_\pm^\theta\left(E+i\epsilon\right)\right|=\infty\iff\theta\in\left\{\frac{\pi}{4},\frac{3\pi}{4}\right\}$. The fact that $J_+^0=J_-^{\frac{\pi}{2}}$ along with (\ref{pert_formula_m_func}) imply that
		\begin{equation}\label{assist_eq_symm_prop}
			m_-^\theta=\frac{m_-^{\frac{\pi}{2}}}{1-\cot\theta m_-^\frac{\pi}{2}}=\frac{m_+^0}{1-\cot\theta m_+^0}
		\end{equation}
		Now, suppose that $\underset{\epsilon\to 0}{\lim}\left|m_\pm^\theta\left(E+i\epsilon\right)\right|=\infty$. Then by (\ref{pert_formula_m_func}) we get that $\underset{\epsilon\to 0}{\lim}m_+^{0}\left(E+i\epsilon\right)=\cot\theta$. Plugging this in (\ref{assist_eq_symm_prop}), we get
		\begin{center}
			$\underset{\epsilon\to0}{\lim}\left|m_-^\theta\left(E+i\epsilon\right)\right|=\begin{cases}
				\left|\frac{\cot\theta}{1-\cot^2\theta}\right| & \theta\in\left[0,\pi\right)\setminus\left\{\frac{\pi}{4},\frac{3\pi}{4}\right\}\\
				\infty & \theta\in\left\{\frac{\pi}{4},\frac{3\pi}{4}\right\}
			\end{cases}$,
		\end{center}
		as required. For any $\theta\in\left[0,\pi\right)$, let
		\begin{center}
			$S_{\pm}^{\theta}=\left\{E\in\mathbb{R}:u_{\theta,E}^\pm\text{ is subordinate as a solution to (\ref{ev_eq1})}\right\}$.
		\end{center}
		Recall that by Theorem \ref{sub_thm_line}, $\mu_s$ is supported on the set $S$ which satisfies $S=\underset{\theta\in\left[0,\pi\right)}{\bigcup}S_-^\theta\cap S_+^\theta$. Now, by Remark \ref{sub_thm_equiv_formulation}, we have
		\begin{center}
			$S_\pm^\theta=\left\{E\in\mathbb{R}:\underset{\epsilon\to0}{\lim}\left|m_\pm^\theta\left(E+i\epsilon\right)\right|=\infty\right\}$,
		\end{center}
		and by what we've shown so far, if $\theta\in\left[0,\pi\right)\setminus\left\{\frac{\pi}{4},\frac{3\pi}{4}\right\}$ then $S_+^\theta\cap S_-^\theta=\emptyset$, and otherwise $S_-^\theta=S_+^\theta$. Thus, we get that $Q$ is supported on $\left\{\frac{\pi}{4},\frac{3\pi}{4}\right\}$. By the properties of the potential we get that for any \mbox{$E\in S_-^{\frac{\pi}{4}}\cap S_+^\frac{\pi}{4}$}, any subordinate solution $u$ of (\ref{ev_eq_line}) satisfies $u\left(n\right)=-u\left(-n+1\right)$, and so \mbox{$\text{Ran}Q\left(\left\{\frac{3\pi}{4}\right\}\right)=\mathcal{H}_{\text{Sym}}^{\perp}$}. Similarly, we get that $\text{Ran}Q\left(\left\{\frac{3\pi}{4}\right\}\right)=\mathcal{H}_\text{Sym}$, which implies the desired result.
	\end{proof}
	\subsection{An operator with pure point spectrum}
	Suppose now that $J$ has pure point spectrum. Recall that $S_\theta\coloneqq S_-^\theta\cap S_-^\theta$ is the set of energies for which there exists a subordinate solution to (\ref{ev_eq_line}) which satisfies (\ref{bc_eq_hl}). In addition, in the case of pure point spectrum, we may assume that a solution is subordinate if and only if it is $\ell^2$ at $\pm\infty$, namely we can take a support $S$ to be the set of eigenvalues of $J$. Then, the set $S_\theta$ which is defined in Section \ref{integral_sec} consists of the eigenvalues of $J$ for which the corresponding eigenvector $\psi$ satisfies
	\begin{center}
		$\psi\left(0\right)\cos\theta+\psi\left(1\right)\sin\theta=0$.
	\end{center}
	For each $\theta\in\left[0,\pi\right)$ such that the set $S_\theta$ is non-empty, we have that $Q\left(S_\theta\right)\neq 0$, and so $Q$ is supported on a countable set $\left\{\theta_n\right\}_n\in\mathbb{N}$. Thus, for any $F\in\mathcal{F}$, where $\mathcal{F}$ is the family of functions given in Theorem \ref{splitting_thm}, the integral formula (\ref{int_eq}) is given by
	\begin{center}
		$\int_I F\left(J_\theta\right)dQ\left(\theta\right)=\underset{n\in\mathbb{N}}{\oplus}F\left(J_{\theta_n}\right)Q\left(\left\{\theta_n\right\}\right)=F\left(J\right)$,
	\end{center}
	where $I=\left[0,\pi\right)$.
	\section{Appendix - Proof of Proposition \ref{L_prop}}
	Our goal in this appendix is to prove Proposition \ref{L_prop}. For any $n\in\mathbb{N}$, we denote \mbox{$\left[-n,n\right]\coloneqq\left\{-n,-\left(n-1\right),\ldots,0,\ldots,n-1,n\right\}$}. We will use the following theorem, which is a special case of \cite[Lemma 3.2]{L}:
	\begin{lemma}\label{Schur_lemma}
		Let $J:\ell^2\left(\mathbb{Z}\right)\to\ell^2\left(\mathbb{Z}\right)$ be a Jacobi operator. Given $n\in\mathbb{N}$ and $z\in\mathbb{C}$, let \mbox{$F\left(n,z\right)\in\mathcal{M}_{2n}\left(\mathbb{C}\right)$} be defined by
		\begin{center}
			$\left(F\left(n,z\right)\right)_{kj}=\langle\delta_k,\left(J-z\right)^{-1}\delta_j\rangle,\,\,\,\,\,\,\,\,\,k,j\in\left[-n,n\right]$.
		\end{center}
		Then $F\left(n,z\right)$ is invertible, and we have
		\begin{equation}\label{F_inv_eq}
			\left(F\left(n,z\right)\right)^{-1}=A\left(n\right)+\diag\left(\frac{1}{m_{-n}\left(z\right)},\ldots,\frac{1}{m_{+n}\left(z\right)}\right)
		\end{equation}
		where $A\left(n\right)\in\mathcal{M}_{2n}$ is given by
		\begin{center}
			$A\left(n,z\right)=\left(\begin{matrix}
				\frac{1}{m_{-n}\left(z\right)} & a_{-n} & 0 & 0 & \cdots & 0 & 0\\
				a_{-n} & b_{-n+1}-z & a_{-n+1} & 0 & \cdots & 0 & 0\\
				0 & a_{-n+1} & b_{-n+2}-z & a_{-n+2} & \cdots & 0 & 0\\
				\vdots & \vdots & \vdots & \vdots & \ddots & \vdots & \vdots\\
				0 & 0 & 0 & 0 & \cdots & b_{n-1}-z & a_{n-1}\\
				0 & 0 & 0 & 0 & \cdots & a_{n-1} & \frac{1}{m_{+n}\left(z\right)}
			\end{matrix}\right)$
		\end{center}
		and $m_{\pm n}$ is the Borel transform of the spectral measure of $\delta_{\pm n}$ w.r.t.\ the restriction of $J$ to $\left\{\pm n,\pm n+1,\ldots\right\}$.
	\end{lemma}
	We now turn to prove Proposition \ref{L_prop}.
	\begin{proof}
		For any $k,j\in\mathbb{Z}$, $\mu_{kj}\ll\mu$. Let $f_{kj}=\frac{d\mu_{kj}}{d\mu}$. It is well known that for $\mu$-almost every $E\in\mathbb{R}$, we have
		\begin{center}
			$f_{kj}\left(E\right)=\underset{\epsilon\to0}{\lim}\frac{\mu_{kj}\left(\left(E-\epsilon,E+\epsilon\right)\right)}{\mu\left(\left(E-\epsilon,E+\epsilon\right)\right)}$
		\end{center}
		and so by Theorem \ref{pol_thm}, there exists a set $\widetilde{S}_{kj}\subseteq\mathbb{R}$ which supports $\mu_s$, such that for any $E\in\widetilde{S}_{kj}$, we have
		\begin{equation}\label{limit_eq_S_prop}
			\underset{\epsilon\to0}{\lim}
			\frac{\mu_{kj}\left(\left(E-\epsilon,E+\epsilon\right)\right)}{\mu\left(\left(E-\epsilon,E+\epsilon\right)\right)}=\underset{\epsilon\to0}{\lim}\frac{M_{kj}\left(E+i\epsilon\right)}{M\left(E+i\epsilon\right)}.
		\end{equation}
		Let $\widetilde{S}=S\cap\underset{k,j\in\mathbb{Z}}{\bigcap}\widetilde{S}_{kj}$. It is easy to see that $\widetilde{S}$ supports $\mu_s$, and that (\ref{limit_eq_S_prop}) holds for any $E\in\widetilde{S}$ and for any $k,j\in\mathbb{Z}$. We denote the limit by $\alpha_{kj}$. By Theorem \ref{supp_thm}, we may assume that for any $E\in\widetilde{S}$,
		\begin{equation}\label{infinite_M_eq}
			\underset{\epsilon\to0}{\lim}\left|M\left(E+i\epsilon\right)\right|=\infty.
		\end{equation}
		Fix $n\in\mathbb{Z}$, $E\in\mathbb{R}$ and $k\in\left[-n.n\right]$. Let $e_k$ be the $k$'th element in the standard basis of $\mathbb{C}^{2n}$. By (\ref{infinite_M_eq}), we have
		\begin{center}
			$0=\underset{\epsilon\to0}{\lim}\frac{1}{\left|M\left(E+i\epsilon\right)\right|}=\underset{\epsilon\to0}{\lim}\left\|\frac{e_k}{\left|M\left(E+i\epsilon\right)\right|}\right\|=\underset{\epsilon\to0}{\lim}\left\|F\left(n,E+i\epsilon\right)^{-1}\left(F\left(n,E+i\epsilon\right)\frac{e_k}{M\left(E+i\epsilon\right)}\right)\right\|$.
		\end{center}
		Now, note that $F\left(n,E+i\epsilon\right)=\left(\begin{matrix}
			M_{1k}\left(E+i\epsilon\right)\\
			\vdots\\
			M_{nk}\left(E+i\epsilon\right)
		\end{matrix}\right)$ and so we get
		\begin{equation}\label{limit_0_eq}
			0=\underset{\epsilon\to0}{\lim}F\left(n,E+i\epsilon\right)^{-1}\left(\begin{matrix}
				\frac{M_{-nk}\left(E+i\epsilon\right)}{M\left(E+i\epsilon\right)}\\
				\vdots\\
				\frac{M_{nk}\left(E+i\epsilon\right)}{M\left(E+i\epsilon\right)}
			\end{matrix}\right)
		\end{equation}
		Taking into account (\ref{F_inv_eq}), we get that for any $j\in\left[-\left(n-1\right),n-1\right]$,
		\begin{equation}\label{eq_2_S_prop}
			\underset{\epsilon\to0}{\lim}a_{j-1}\frac{M_{\left(j-1\right)k}\left(E+i\epsilon\right)}{M\left(E+i\epsilon\right)}+a_j\frac{M_{\left(j+1\right)k}\left(E+i\epsilon\right)}{M\left(E+i\epsilon\right)}+\frac{M_{jk}\left(E+i\epsilon\right)}{M\left(E+i\epsilon\right)}\left(b_j-\left(E+i\epsilon\right)\right)=0,
		\end{equation}
		Define $\widetilde{u}$ on $\left[-n,n\right]$ by $\widetilde{u}\left(j\right)=\alpha_{kj}$. By (\ref{eq_2_S_prop}), we get that for any $j\in\left[-\left(n-1\right),n-1\right]$,
		\begin{center}
			$a_{j-1}\widetilde{u}\left(j-1\right)+a_j\widetilde{u}\left(j+1\right)+b_j\widetilde{u}\left(j\right)=0$.
		\end{center}
		This implies that there exists a solution $u$ of (\ref{ev_eq_line}) such that $u|_{\left[-n,n\right]}=\widetilde{u}|_{\left[-n,n\right]}$. We claim that $u$ is either a subordinate solution or $u$ vanishes identically. We will show this for $u|_{\left[n,\infty\right)}$, the proof for $u|_{\left(-\infty,-n\right]}$ is similar. Suppose $\alpha_n=0$. In addition, suppose that there exists a sequence $\left(\epsilon_n\right)_{n=1}^{\infty}$ such that $\underset{n\to\infty}{\lim}\epsilon_n=0$ and
		\begin{center}
			$\underset{n\to\infty}{\lim}m_{+n}\left(E+i\epsilon_n\right)\neq 0$.
		\end{center}
		By (\ref{F_inv_eq}) and (\ref{limit_0_eq}), this implies that $a_{n-1}\alpha_{n-1}=0$. Thus, $u$ is a solution of (\ref{ev_eq_line}) which satisfies $u\left(n-1\right)=u\left(n\right)=0$, which implies that it vanishes identically. If no such sequence $\left(\epsilon_n\right)_{n=1}^{\infty}$ exists, then $\underset{\epsilon\to0}{\lim}m_{+n}\left(E+i\epsilon\right)=0$. The half-line subordinacy now implies (see Remark \ref{sub_thm_equiv_formulation}) that any non-zero solution which satisfies $u\left(n\right)=0$ is subordinate. This concludes the proof for the case $\alpha_n=0$. Suppose now that $\alpha_n\neq 0$. Again, by (\ref{F_inv_eq}) and (\ref{limit_0_eq}), we get that the limit
		\begin{center}
			$\underset{\epsilon\to0}{\lim}\frac{1}{m_{+n}\left(E+i\epsilon\right)}$
		\end{center}
		exists and is real, and so there exists $\theta\in\left[0,\pi\right)$ such that $\underset{\epsilon\to0}{\lim}m_{+n}\left(E+i\epsilon\right)=\cot\theta$. By (\ref{JL_hl_sub_thm}), we get $u|_{\left[n,\infty\right)}$ is subordinate if and only if $a_{n-1}u\left(n-1\right)\cos\theta+u\left(n\right)\sin\theta=0$. By (\ref{F_inv_eq}) and (\ref{limit_0_eq}), we get
		\begin{equation}
			a_{n-1}\alpha_{n-1}+\tan\theta\alpha_n=\underset{\epsilon\to0}{\lim}a_{n-1}\frac{M_{\left(n-1\right)k}\left(E+i\epsilon\right)}{M\left(E+i\epsilon\right)}+\frac{M_{nk}\left(E+i\epsilon\right)}{M\left(E+i\epsilon\right)}\frac{1}{m_{+n}\left(E+i\epsilon\right)}=0
		\end{equation}
		which implies that $u$ is subordinate, as required. For any $E\in\widetilde{S}$, let us denote by $\Gamma\left(E\right)$ the $\mathbb{Z}\times\mathbb{Z}$ matrix which is given by $\Gamma\left(E\right)_{kj}=\alpha_{kj}$. We conclude that for any $j\in\mathbb{Z}$ and for any $E\in\widetilde{S}$, the function $u:\mathbb{Z}\to\mathbb{C}$ which is defined by $u\left(k\right)=\Gamma\left(E\right)_{kj}$ is either a subordinate solution of (\ref{ev_eq_line}) or identically zero. Thus, for any $n\in\mathbb{Z}$, the matrix $\Gamma_n\left(E\right)\in M_{2n}\left(\mathbb{C}\right)$ which is given by restricting $\Gamma\left(E\right)$ to $\left[-n,n\right]\times\left[-n,n\right]$ has rank one. We claim that it is also self-adjoint. For any $\epsilon>0$, let $I_\epsilon^E\coloneqq\left(E-\epsilon,E+\epsilon\right)$. For any $k,j\in\left[-n,n\right]$, we have
		\begin{center}
			$\Gamma\left(E\right)_{kj}=\underset{\epsilon\to0}{\lim}\frac{\mu_{kj}\left(I_\epsilon^E\right)}{\mu\left(I_\epsilon^E\right)}=\underset{\epsilon\to0}{\lim}\frac{1}{\mu\left(I_\epsilon^E\right)}\langle P\left(I_\epsilon^E\right)\delta_k,\delta_j\rangle=\underset{\epsilon\to0}{\lim}\frac{1}{\mu\left(I_\epsilon^E\right)}\langle\delta_k,P\left(I_\epsilon^E\right)\delta_j\rangle\overset{\left(\star\right)}{=}\underset{\epsilon\to0}{\lim}\frac{1}{\mu\left(I_\epsilon^E\right)}\langle P\left(I_\epsilon^E\right)\delta_j,\delta_k\rangle=\underset{\epsilon\to0}{\lim}\frac{\mu_{jk}\left(I_\epsilon^E\right)}{\mu\left(I_\epsilon^E\right)}=\Gamma\left(E\right)_{jk}$.
		\end{center}
		Let us justify $\left(\star\right)$. First, note that $P\left(I_\epsilon^E\right)=\underset{n\to\infty}{\lim}q_n\left(J\right)$, where $\left(q_n\right)_{n=1}^{\infty}\subseteq\mathbb{R}\left[X\right]$. This implies that $\langle\delta_k,P\left(I_\epsilon^E\right)\delta_j\rangle=\underset{n\to\infty}{\lim}\langle\delta_k,q_n\left(J\right)\delta_j\rangle$, and this implies that $\langle\delta_k,P\left(I_\epsilon^E\right)\delta_j\rangle\in\mathbb{R}$ as a limit of a sequence of real numbers, as required.
		
		Since $\Gamma_n\left(E\right)$ is a self-adjoint matrix of rank-one, there exists $u_n\in\mathbb{C}^n$ such that \mbox{$\Gamma\left(E\right)_{kj}=u_n\left(k\right)u_n\left(j\right)$}. Furthermore, it is clear that for every $n,m\in\mathbb{N}$ such that $n>m$, $u_n|_{\left[-m,m\right]}=u_m$ and so there exists $u_E:\mathbb{Z}\to\mathbb{C}$ such that for any $k,j\in\mathbb{Z}$, $\left(\Gamma\left(E\right)\right)_{kj}=u_E\left(k\right)u_E\left(j\right)$. Finally, by what we have shown, $u$ must be a subordinate solution of (\ref{ev_eq_line}), and
		\begin{center}
			$u_E\left(0\right)^2+u_E\left(1\right)^2=\Gamma\left(E\right)_{00}+\Gamma\left(E\right)_{11}=\underset{\epsilon\to0}{\lim}\frac{\mu_{00}\left(I_\epsilon^E\right)}{\mu\left(I_\epsilon^E\right)}+\frac{\mu_{11}\left(I_\epsilon^E\right)}{\mu\left(I_\epsilon^E\right)}=\underset{\epsilon\to0}{\lim}\frac{\mu\left(I_\epsilon^E\right)}{\mu\left(I_\epsilon^E\right)}=1$,
		\end{center}
		as required.
	\end{proof}
	
\end{document}